\documentclass[preprint]{imsart}

\usepackage[numbers]{natbib}
\usepackage[T1]{fontenc}
\usepackage[latin1]{inputenc}
\usepackage[english]{babel}
\usepackage{amsmath,amsthm,amssymb,amsfonts,epic,latexsym,hyperref
}

\usepackage{lipsum}
\usepackage{rotating}
\usepackage{dsfont}
\usepackage{geometry}
\usepackage{stmaryrd}
\usepackage{ifthen}
\usepackage[nottoc]{tocbibind}
\usepackage{verbatim}
\usepackage{mathtools}
\usepackage{tikz}

\newcommand\un{\mathds{1}}
\newcommand\eps{\varepsilon}
\renewcommand\phi{\varphi}
\newcommand\pro[1]{\mathbb{P}\left(#1\right)}
\newcommand\esp[1]{\mathbb{E}\left[#1\right]}

\newcommand\uno[1]{\un_{\left\{#1\right\}}}
\newcommand\1{\un}

\newcounter{comptetape}
\newtheorem{thm}{Theorem}[section]
\newtheorem{prop}[thm]{Proposition}

\newtheorem{lem}[thm]{Lemma}
\newtheorem{rem}[thm]{Remark}
\newtheorem{defi}[thm]{Definition}
\newtheorem{ex}[thm]{Example}

\newcommand\n{\mathbb{N}}
\renewcommand\r{\mathbb{R}}

\renewcommand\ll{\left}
\newcommand\rr{\right}

\begin{document}

\begin{frontmatter}

\title{A convergence criterion for systems of point processes from the convergence of their stochastic~intensities.}

\runtitle{Convergence of point processes}

\begin{aug}

\author{\fnms{Xavier} \snm{Erny}\ead[label=e4]{xavier.erny@univ-evry.fr}}

\address{
  Universit\'e Paris-Saclay, CNRS, Univ Evry, Laboratoire de Math\'ematiques et Mod\'elisation d'Evry, 91037, Evry, France
}

\runauthor{X. Erny}

 \end{aug}

\begin{abstract}
We study systems of simple point processes that admit stochastic intensities. We represent these point processes as thinnings of Poisson measures and are interested in a convergence result of such systems. This result states that, if the stochastic intensities of the limit point processes are independent of the underlying Poisson measures, the convergence in distribution in Skorohod topology of the stochastic intensities implies the same convergence for the point processes.
\end{abstract}



\end{frontmatter}

\section{Introduction}

In this paper we consider systems of point processes admitting stochastic intensities. Processes of these type arise naturally in the study of particle systems, such as modeling of neural activity (see e.g. \cite{okatan_analyzing_2005} or \cite{reynaud-bouret_goodness--fit_2014}) or financial data (see e.g. \cite{bauwens_modelling_2009} or \cite{lu_high_2018}).
Note that such processes can always be written as thinnings of Poisson measures.

The natural question is: when does the convergence in distribution in Skorohod space of systems of such processes follow from the convergence of their intensities? In this paper we study this question. 
Our main result, Theorem~\ref{convergencepointprocess}, roughly states  that if the intensities of point processes converge in distribution and their limits are independent of the underlying Poisson measures, then the systems of point processes converge as well.


To give the formal statement of Theorem~\ref{convergencepointprocess}, denote by $D(\r_+,\r)$ Skorohod space, $\mathcal{M}$ the space of locally finite measures endowed with the topology of vague convergence and $\mathcal{N}$ the subspace of locally finite simple point measures.

\begin{thm}
\label{convergencepointprocess}
Let $\bar{Y}^k$ and $Y^{N,k}$ $(N,k\in\n^*)$ be $D(\r_+,\r_+)$-valued random variables. Let $\ll(\pi^k\rr)_{k\in\n^*}$ and $\ll(\bar{\pi}^k\rr)_{k\in\n^*}$ be i.i.d. families of Poisson measures on $\r_+\times\r_+$ having Lebesgue intensity. Let $Z^{N,k}$ and $\bar{Z}^k$ be point processes defined as follows
\begin{equation*}
Z^{N,k}_t:=\int_{[0,t]\times\r_+}\uno{z\leq Y^{N,k}_{s-}}d\pi^k(s,z),~~\bar{Z}^k_t:=\int_{[0,t]\times\r_+}\uno{z\leq \bar{Y}^k_{s-}}d\bar{\pi}^k(s,z),~~~k\geq 1.
\end{equation*}
Assume that, for every $n\geq 1,(Y^{N,1},\pi^1,...,Y^{N,n},\pi^n)$ converges in distribution to $(\bar Y^{1},\bar\pi^1,...,\bar Y^{n},\bar\pi^n)$ in $(D(\r_+,\r)\times\mathcal{N})^n,$ and that, for each $k\geq 1,$ $\bar Y^k$ is independent of $\bar\pi^k.$

Then, for any $n\geq 1$, $\ll(Z^{N,k}\rr)_{1\leq k\leq n}$ converges to $\ll(\bar{Z}^k\rr)_{1\leq k\leq n}$ in distribution in $D(\r_+,\r^n).$ In particular, $\ll(Z^{N,k}\rr)_{k\geq 1}$ converges to $\ll(\bar{Z}^k\rr)_{k\geq 1}$ in distribution in $D(\r_+,\r)^{\n^*}$ endowed with the product topology.
\end{thm}

\begin{rem}
In the statement of Theorem~\ref{convergencepointprocess}, we need to guarantee the following property: Poisson random measures are $\mathcal{N}-$valued random variables. This is a direct consequence of Theorem~2.6.III.(ii) of \cite{daley_introduction_2003} and of the definition of Poisson measures (see Definition~\ref{poisson}).
\end{rem}

\begin{rem}
According to Lemma~4 of~\cite{bremaud_stability_1996}, a point process $Z$ having stochastic intensity $\ll(Y_{s-}\rr)_{s\geq 0}$ (where $Y$ is a c\`adl\`ag process) can always be written in the form of Theorem~\ref{convergencepointprocess}.
\end{rem}

Let us note that, in Theorem~\ref{convergencepointprocess}, the processes $Y^{N,k}$ are not assumed to be independent of the Poisson measures $\pi^k$. Otherwise, the proof would be straightforward by conditioning by $Y^{N,k}$.
Let us also note that the condition of independence of the limiting intensities $\bar Y^k$  from Poisson measures $\bar \pi^k$ is often satisfied and natural in many examples of application. It holds for example when the limiting intensities are deterministic, or when they are functionals of Brownian motions with respect to the same filtration as $\bar \pi^k$.

If the processes  $Y^{N,k}$ are semimartingales, 
 the result of Theorem~\ref{convergencepointprocess} can 
 follow from Theorem~IX.4.15 of \cite{jacod_limit_2003},
 provided the convergence of the characteristics of the corresponding semimartingales holds. But we do not assume the semimartingale structure for the intensities~$Y^{N,k}$ in Theorem~\ref{convergencepointprocess}. 
  
 When the limiting intensities are deterministic,
 Theorem~\ref{convergencepointprocess} can be compared to Theorem~1 of \cite{brown_martingale_1978} that states that the convergence of point processes is implied by the pointwise convergence in distribution of their compensators (i.e. for each $t\geq 0,$ the compensator at time~$t$ converges in distribution in~$\r$). In \cite{brown_martingale_1978}, Theorem~1 holds when the compensator of the limit point process is a deterministic function, whereas in Theorem~\ref{convergencepointprocess}, the limit point processes have stochastic intensities.
  
The idea of the proof of Theorem~\ref{convergencepointprocess} consists in writing the point processes $Z^{N,k}$ and $\bar Z^k$ as a function $\Phi$ of respectively $Y^{N,k},\pi^k$ and $\bar Y^k,\bar\pi^k$ ($k\geq 1$). Then, knowing that, $(Y^{N,k},\pi^k)$ converges in distribution to $(\bar Y^k,\bar \pi^k)$, we can use Skorohod representation theorem to assume that this convergence is almost sure. Finally, proving that $\Phi$ is almost surely continuous at~$(\bar Y^k,\pi^k)$ concludes the proof.

The organisation of the paper is the following: 
in Section \ref{secexamples} we give some examples of applications of Theorem~\ref{convergencepointprocess}, then in Section~\ref{sectionpoisson} we recall 
some classical properties of Poisson measures and of the vague convergence. Finally, Section~\ref{sectionconvergence} is devoted to the proof of Theorem~\ref{convergencepointprocess}

 \section {Examples}
 \label{secexamples}
 
  We start this section with an example of system of point processes based on \cite{erny_mean_2019}. In this example  the stochastic intensities are in fact semimartingales. We illustrate both methods of proof of the convergence of this system: using Theorem~IX.4.15 of \cite{jacod_limit_2003} about semimartingale's convergence and using our Theorem~\ref{convergencepointprocess}.
\begin{ex}
Let $\alpha>0$ and $\pi^k$ ($k\geq 1$) be independent Poisson measures on $\r_+\times\r_+\times\r$ with intensity $dt\cdot dz\cdot d\nu(u)$ where $\nu$ is the centered normal distribution with variance one. Define $X^N$ as solution of
$$dX^N_t=-\alpha X^N_tdt+\frac{1}{\sqrt{N}}\sum_{j=1}^N\int_{\r_+\times\r}u\uno{z\leq 1+(X^N_{t-})^2}d\pi^k(t,z,u),$$
and
$$Z^{N,i}_t:=\int_{[0,t]\times\r_+}\uno{z\leq 1+(X^N_{s-})^2}d\pi^i(s,z,u).$$

It was shown in Theorem~1.4  of \cite{erny_mean_2019} that $X^N$ converges in distribution in $D(\r_+,\r)$  to $\bar X$, where $\bar X$ is solution of

$$d\bar X_t=-\alpha\bar X_tdt+\sqrt{1+\bar X_t^2}dW_t,$$
with $W$ some one-dimensional standard Brownian motion.
Define also
$$\bar Z^i_t=\int_{[0,t]\times\r_+}\uno{z\leq 1+\bar X_{s-}^2}d\bar\pi^i(s,z,u),$$
where $\bar\pi^i$ ($i\geq 1$) are independent Poisson measures on $\r_+^2$ having Lebesgue intensity.

In order to show the convergence of $(Z^{N,i})_{1\leq i\leq n}$  in distribution  in $D(\r_+,\r^n)$ using Theorem~IX.4.15 of \cite{jacod_limit_2003},
we have to consider the semimartingale $(X^N,Z^{N,1},...,Z^{N,n})$ and show that its caracteristics converge to those of $(\bar X,\bar Z^{1},...,\bar Z^{n}).$

An alternative proof of the convergence of $(Z^{N,i})_{1\leq i\leq n}$ to $(\bar Z^i)_{1\leq i\leq n}$ relies on the Theorem~\ref{convergencepointprocess}. 
Indeed, following Theorem~1.4 of \cite{erny_mean_2019},
 $X^N$ converges in distribution in $D(\r_+,\r)$ to $\bar X,$ and, being adapted to the same filtration,  the Brownian motion~$W$ is independent of the Poisson measures $\bar\pi^i$ ($i\geq 1$) (see Theorem~II.6.3 of \cite{ikeda_stochastic_1989}).
\end{ex}

Theorem~\ref{convergencepointprocess} allows us to consider processes that are not semimartingales, such as Hawkes processes and Volterra processes. Since the stochastic intensities of Hawkes processes are not, in general, semimartingales, Theorem~\ref{convergencepointprocess} can be interesting to show the convergence of Hawkes processes, provided that one can show the convergence of their stochastic intensities. Let us give an example of application of Theorem~\ref{convergencepointprocess} in this case. The example is based on Examples~7.3 and~7.4 of \cite{abi_jaber_weak_2019}.
\begin{ex}
Let us consider $K(t):=t^\gamma$ for some $\gamma>0$, $K^N(t):=K(t/N)$ and some Poisson random measure $\pi$ on $\r_+^2$ having Lebesgue intensity. Let $X^N$ satisfies
$$X_t^N=\int_{[0,t]\times\r_+}K^N(t-s)\uno{z\leq\ll|X^N_{s-}\rr|}d\pi(s,z)-\int_0^tK^N(t-s)\ll|X^N_s\rr|ds.$$
Theorem~7.2 of \cite{abi_jaber_weak_2019} implies that the sequence of processes $(\tilde X^N_t)_{t\geq 0}=(N^{-1}X^N_{Nt})_{t\geq 0}$ has converging subsequences (in distribution in the topology $L^2_{loc}$), and that every limit process $(\bar X_t)_{t\geq 0}$ satisfies
\begin{equation}
\label{volterra}
\bar X_t=\int_0^tK(t-s)\sqrt{|\bar X_s|}dB_s,
\end{equation}
for some standard Brownian motion~$B$. Besides, one can prove with standard arguments, the tightness of $(\tilde X^N)_N$ in Skorohod topology. Then, we can consider a subsequence of $(\tilde X^N)_N$ that converges in distribution in the topology of $L^2_{loc}$ and in Skorohod topology. The limit for both topologies is necessarily the same on Skorohod space. Indeed, let $\hat x$ be the limit for $L^2_{loc}$ topology and $\check x$ be the limit for Skorohod topology of a sequence of c\`adl\`ag functions $(x^n)_n$. The $L^2_{loc}$ convergence implies that $x^n_t$ converges to $\hat x_t$ for Lebesgue-a.e. $t\geq 0,$ whence, the convergence for all continuity point~$t$ of $\hat x$. Besides, the convergence in Skorohod topology implies the convergence $x^n_t$ to $\check x_t$ for every continuity point~$t$ of $\check x.$ This implies that $\hat x$ and $\check x$ have the same continuity points, and that $\check x=\hat x.$

This implies the convergence of (a subsequence of) $Y^N:=|\tilde X^N|$ to $\bar Y:=|\bar X|$ in Skorohod topology. Moreover, the Brownian motion can be shown to be necessarily independent of the Poisson measure~$\pi$ (using Theorem~II.6.3 of \cite{ikeda_stochastic_1989}). Then, Theorem~\ref{convergencepointprocess} implies the convergence in distribution in Skorohod topology of $Z^N_t:=\int_{[0,t]\times\r_+}\uno{z\leq Y^N_{s-}}d\pi(s,z)$ to the point process $\bar Z_t:=\int_{[0,t]\times\r_+}\uno{z\leq \bar Y_{s-}}d\bar\pi(s,z),$ where $\bar\pi$ is independent of $\bar Y$. To the best of our knowledge, there is no classical way to prove this convergence.
\end{ex}
As it was mentioned in Introduction, 
if the stochastic intensities of the limit point processes are deterministic, then the hypothesis of Theorem~\ref{convergencepointprocess} are satisfied. Let us give a practical example based on \cite{delattre_hawkes_2016}.
\begin{ex}
Consider any locally integrable $K:\r_+\rightarrow\r$ and any Lipschitz continuous function~$f:\r\rightarrow\r_+$. Define $X^N$ as solution of
$$X^N_t=\frac1N\sum_{j=1}^N\int_{[0,t]\times\r_+}K(t-s)\uno{z\leq f(X^{N}_{s-})}d\pi^j(s,z),$$
where $\pi^j$ ($j\geq 1$) are independent Poisson measures on $\r_+^2$ with Lebesgue intensity, and $\bar X$ as the deterministic solution of
$$\bar X_t=\int_0^tK(t-s)f(\bar X_s)ds.$$
Let
$$Z^{N,i}_t=\int_{[0,t]\times\r_+}\uno{z\leq f(X^N_{s-})}d\pi^i(s,z)\textrm{ and }\bar Z^i_t=\int_{[0,t]\times\r_+}\uno{z\leq f(\bar X_{s-})}d\pi^i(s,z).$$

Then, Theorem~8 of \cite{delattre_hawkes_2016} states that, for $i\geq 1,$ for all $T\geq 0,$
$$\esp{\underset{0\leq t\leq T}{\sup}\ll|Z^{N,i}_t-\bar Z^i_t\rr|}\leq C_T N^{-1/2},$$
for some constant $C_T.$ In other words, $Z^{N,i}$ converges to $\bar Z$ in a $L^1-$sense. As $\bar X$ is a deterministic function, Theorem~1 of \cite{brown_martingale_1978}, and a fortiori Theorem~\ref{convergencepointprocess}, can be used to show the (weaker) convergence in distribution in $D(\r_+,\r)$ of $Z^{N,i}$ to $\bar Z,$ provided $X^N$ converges to $\bar X.$
\end{ex}
%
%

\section{Classical properties of Poisson measures and the vague convergence}\label{sectionpoisson}

Let us begin with the usual definition of random measures and Poisson measures. We restrict this definition to the space $\r_+^2$ since we only need this space in the paper, but Definition~\ref{poisson} can be generalized to any measurable space. In the rest of the paper, $\r_+^2$ is always endowed with the Borel sigma algebra~$\mathcal{B}(\r_+^2)$.
\begin{defi}
\label{poisson}
A locally finite random measure on $\r_+^2$ is a $\mathcal{M}-$valued random variable, where $\mathcal{M}$ is endowed with the $\sigma-$algebra generated by the functions $\pi\in\mathcal{N}\mapsto\pi(B)$ ($B\in\mathcal{B}(\r_+^2)$).

A Poisson measure on $\r_+^2$ is a locally finite random measure $\pi$ satisfying:
\begin{itemize}
\item for all $B\in\mathcal{B}(\r_+^2),$ $\pi(B)$ follows a Poisson distribution,
\item for all $n\in\n^*,$ for all disjoint sets $B_1,\hdots,B_n\in\mathcal{B}(\r_+^2),$ the variables $\pi(B_i)$ ($1\leq i\leq n$) are independent.
\end{itemize}

The function $\mu : B\in\mathcal{B}(\r_+^2)\mapsto\esp{\pi(B)}$ is a measure on $\r_+^2$ that is called the intensity of $\pi.$
\end{defi}

\begin{rem}
In Definition~\ref{poisson}, we can consider Poisson distribution with parameter infinity. A Poisson variable with parameter infinity is a random variable~$X$ satisfying $X=+\infty$~a.s.
\end{rem}

Let us begin with an elementary lemma.

\begin{lem}
\label{mesurenulle}
Let $D\in\mathcal{B}(\r_+^2)$, $\mu$ be a (deterministic) measure on $\r_+^2,$ and $\pi$ be a Poisson measure on $\r_+^2$ with intensity $\mu.$ If $\mu(D)=0,$ then, a.s. $\pi(D)=0.$ 
\end{lem}

\begin{proof}
By definition, $\pi(D)$ is a Poisson variable with parameter $\mu(D)=0.$
\end{proof}

Let us recall an important representation result of Poisson measures as marked renewal processes.

\begin{thm}
\label{representationpoissonmeasure}
Let $\lambda>0$, $\pi$ be a Poisson measure on $\r_+^2$ with Lebesgue intensity. Let $(T_i,Z_i)$ ($i\geq 1$) be the atoms of $\pi_{|\r_+\times[0,\lambda]}$, lexicographically ordered, and $T_0:=0$. Then
\begin{itemize}
\item $T_i,Z_i$ ($i\geq 1$) are real-valued random variables,
\item the variables $T_{i+1}-T_i$ ($i\geq 0$) are i.i.d. following the exponential distribution with parameter~$\lambda,$
\item $Z_i$ ($i\geq 1$) are i.i.d. following uniform distribution on~$[0,\lambda]$,
\item the families $(T_i)_{i\geq 1}$ and $(Z_i)_{i\geq 1}$ are independent.
\end{itemize}
\end{thm}

This result allows us to prove another classical property of the Poisson measures.

\begin{lem}
\label{auplusun}
Let $\pi$ be a Poisson measure on $\r_+^2$ with Lebesgue intensity. Then,
$$\pro{\forall t\geq 0,\pi(\{t\}\times\r_+)\leq 1}=1.$$
\end{lem}

\begin{proof}
\begin{align*}
\pro{\exists t,\pi(\{t\}\times\r_+)\geq 2}=&\pro{\bigcup_{n\in\n^*}\ll\{\exists t,\pi_{|\r_+\times[0,n]}(\{t\}\times\r_+)\geq 2\rr\}}\\
=&\underset{n\rightarrow\infty}{\lim}\pro{\exists t,\pi_{|\r_+\times[0,n]}(\{t\}\times\r_+)\geq 2}
\end{align*}

Let $n\in\n^*.$ It is then sufficient to show that 
$$\pro{\exists t,\pi_{|\r_+\times[0,n]}(\{t\}\times\r_+)\geq 2}=0.$$

To this end, we use Theorem~\ref{representationpoissonmeasure} to write the atoms of $\pi_{|\r_+\times[0,n]}$ as $(T_i,Z_i)_{i\geq 1}.$ Then, we have
\begin{align*}
\pro{\exists t,\pi_{|\r_+\times[0,n]}(\{t\}\times\r_+)\geq 2}=&\pro{\exists i\geq 1,T_{i}=T_{i+1}}\\
\leq &\sum_{i\geq 1}\pro{T_i=T_{i+1}}=0
\end{align*}
since $T_{i+1}-T_i$ follows the exponential distribution with parameter $n$ ($i\geq 1$).
\end{proof}

We end this section with a result claiming that the vague convergence of locally finite point measures implies the convergence of their atoms.

\begin{prop}
\label{cvvague}
Let $P^k$ ($k\in\n$) and $P$ be locally finite simple point measures on $\r_+^2$ such that $P^k$ converges vaguely to $P.$ Let $T>0,M>0$ such that $P(\partial [0,T]\times[0,M])=0$. Denote $n_k:=P^k([0,T]\times[0,M]), n:=P([0,t]\times [0,M])$ and $(t^k_i,z^k_i)_{1\leq i\leq n_k}$ (resp. $(t_i,z_i)_{1\leq i\leq n}$) the atoms of $P^k_{|[0,T]\times[0,M]}$ (resp. $P_{|[0,T]\times[0,M]}$).

Then, for $k$ large enough, $n_k=n$, and there exists a sequence of permutation $(\sigma^k)_k$ of $\llbracket 1,n\rrbracket$ such that, for all $1\leq i\leq n,$ $t^k_{\sigma^k(i)}$ (resp. $z^k_{\sigma^k(i)}$) converges to $t_i$ (resp. $z_i$) as $k$ goes to infinity.  
\end{prop}

\begin{proof}
To begin with, Proposition~A2.6.II.(iv) of \cite{daley_introduction_2003} implies that $n_k$ converges to $n$ as $k$ goes to infinity. As $n_k$ ($k\in\n$) and $n$ are integers, this implies that $n_k=n$ for $k$ large enough.

To show the convergence of $t^k_i$ and $z^k_i,$ ($1\leq i\leq n$) let us fix some $\eps>0.$ Then, for each $1\leq i\leq n,$ consider an open ball $B_i$ centered on $(t_i,z_i)$ of radius smaller than $\eps$ (for the supremum norm) such that, for all $i\neq j,B_i\cap B_j=\varnothing.$

Using again Proposition~A2.6.II.(iv) of \cite{daley_introduction_2003}, we know that $P^k(B_i)$ converges to $P(B_i)=1$ ($1\leq i\leq n$). Since $P^k$ are point measures, this implies that for all $1\leq i\leq n,$ $P^k(B_i)=1$ for $k$ large enough. As the sets $B_i$ ($1\leq i\leq n$) are disjoint, there exists a permutation $\sigma^k$ such that $(t^k_{\sigma^k(i)},z^{k}_{\sigma^k(i)})\in B_i.$ Hence, for $k$ large enough, for all $1\leq i\leq n,$ $|t^k_{\sigma^k(i)}-t_i|\leq \eps$ and $|z^k_{\sigma^k(i)}-z_i|\leq \eps.$
\end{proof}

\section{Proof of Theorem~\ref{convergencepointprocess}}\label{sectionconvergence}

This section is dedicated to prove Theorem~\ref{convergencepointprocess}. Let us begin with an important result.

\begin{thm}
\label{continuitypointprocess}
Let $\Phi:D(\r_+,\r_+)^m\times\mathcal{N}^m\rightarrow D(\r_+,\r^m)$ be defined as
$$\Phi(x,\pi)_t:=\ll(\int_{[0,t]\times\r_+}\uno{z\leq x^j_{s-}}d\pi^j(s,z)\rr)_{1\leq j\leq m}.$$
Let $(x,\pi)\in D(\r_+,\r_+)^m\times\mathcal{N}^m$. A sufficient condition for $\Phi$ to be continuous at~$(x,\pi)$ is:
\begin{itemize}
\item[(a)] for each $1\leq j\leq m,$ for every $t\geq 0,$ $\pi^j(\{t\}\times\r_+)\leq 1,$
\item[(b)] for each $1\leq j\leq m,$ for every $t\geq 0,$ if $\pi^j(\{t\}\times\r_+)=1,$ then, for all $i\neq j,$ $\pi^i(\{t\}\times\r_+)=0,$
\item[(c)] for each $1\leq j\leq m,$ for every $t\geq 0$ such that $\pi^j(\{t\}\times\r_+)=1,~x^j$ is continuous at $t,$
\item[(d)] for each $1\leq j\leq m$, $\pi^j\ll(\{(t,x^j_{t-})~:~t\geq 0\}\rr)=0.$
\end{itemize}
\end{thm}

Before proving Theorem~\ref{continuitypointprocess}, let us point out that, in general, $\Phi$ is not continuous at every point of $D(\r_+,\r_+)^m\times\mathcal{N}^m$. This is shown in Example~\ref{pascontinue}, where hypothesis~$(d)$ is not satisfied.

\begin{ex}
\label{pascontinue}
Let us consider the point measure $\pi=\delta_{(1,1)}$ and the constant function $x:t\in\r_+\mapsto 1$. In addition, we consider the functions $x^n$ defined as in Figure~\ref{xN} below. Obviously, $||x-x^n||_\infty=1/n$, but $\Phi(x,\pi)_t=\uno{t\geq 1}$ and $\Phi(x^n,\pi)=0$. In other words, $x^n$ converges uniformly to $x$, but $\Phi(x^n,\pi)$ does not converge to $\Phi(x,\pi)$ for non-trivial topologies.
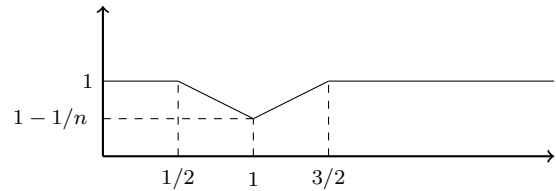
\begin{figure}[ht!]
\centering
\begin{tikzpicture}
\draw[->,thick](0,0)--(0,2);
\draw[->,thick](0,0)--(6,0);
\draw(0,1)--(1,1);
\draw(1,1)--(2,0.5);
\draw(2,0.5)--(3,1);
\draw(3,1)--(6,1);
\draw[dashed](2,0)--(2,0.5);
\draw[dashed](1,0)--(1,1);
\draw[dashed](3,0)--(3,1);
\draw[dashed](0,0.5)--(2,0.5);
\node at (1,-0.3) {$1/2$};
\node at (3,-0.3) {$3/2$};
\node at (2,-0.3) {$1$};
\node at (-0.7,0.5) {$1-1/n$};
\node at (-0.2,1) {$1$};
\end{tikzpicture}
\caption{Graph of $x^n$}
\label{xN}
\end{figure}
\end{ex}

The proof of Theorem~\ref{continuitypointprocess} uses the following Lemmas about the convergence in Skorohod space.

\begin{lem}
\label{convergenceponctuel}
Let $(x_N)_N$ be a sequence of $D(\r_+,\r)$ converging to some $x\in D(\r_+,\r)$, and $\left(t_N\right)_N$ be a sequence converging to some $t>0$. If $x$ is continuous at $t$, then $x_N(t_N-)\rightarrow x(t)$.
\end{lem}

\begin{proof}
Let $T>t$ such that $x$ is continuous at~$T$. By Theorem~16.2 of \cite{billingsley_convergence_1999}, $x_N$ converges to $x$ in $D([0,T],\r)$. Consequently, there exists a sequence of continuous increasing bijective functions $\lambda_N:[0,T]\to[0,T]$ such that $\ll|\ll|\lambda_N-Id\rr|\rr|_{\infty,[0,T]}$ and $\ll|\ll|x_N-x\circ\lambda_N\rr|\rr|_{\infty,[0,T]}$ vanish as $N$ goes to infinity. Then, as $\lambda_N$ is continuous,
\begin{align*}
\ll|x(t)-x_N(t_N-)\rr|\leq& \ll|x(t)-x(\lambda_N(t_N))\rr|+\ll|x(\lambda_N(t_N-))-x_N(t_N-)\rr|\\
\leq& \ll|x(t)-x(\lambda_N(t_N))\rr|+||x\circ\lambda_N-x_N||_{\infty,[0,T]}
\end{align*}
vanishes as $N$ goes to infinity since
$$\ll|\lambda_N(t_N)-t\rr|\leq \ll|\lambda_N(t_N)-t_N\rr|+\ll|t_N-t\rr|\leq \ll|\ll|\lambda_n-Id\rr|\rr|_{\infty,[0,T]}+\ll|t_N-t\rr|.$$
\end{proof}

\begin{lem}
\label{critereconvergence}
Let $T>0$, $k\in\n^*$, $n_i\in\n^*,$ and consider increasing sequences $0=t_{i,0}<t_{i,1}<\hdots<t_{i,n_i-1}<t_{i,n_i}=T$, $0=t_{i,0}^N<t_{i,1}^N<\hdots<t_{i,n_i^N-1}^N<t_{i,n_i^N}^N=T$ ($1\leq i\leq k$). We define the functions $g,g_N\in D([0,T],\r^k)$ by
\begin{equation*}
\left\{\begin{array}{l}
g(t)=\left(\sum_{j=0}^{n_i-1}\un_{[t_{i,j},t_{i,j+1}[}(t)j\right)_{1\leq i\leq k}\textrm{ for }t\in[0,T[,\\
g(T)=\left(n_i-1\right)_{1\leq i\leq k},
\end{array}\right.
\end{equation*}
and
\begin{equation*}
\left\{\begin{array}{ll}
g_N(t)=\left(\sum_{j=0}^{n_i^N-1}\un_{[t_{i,j}^N,t_{i,j+1}^N[}(t)j\right)_{1\leq i\leq k}\textrm{ for }t\in[0,T[,\\
g_N(T)=\left(n_i^N-1\right)_{1\leq i\leq k}.
\end{array}\right.
\end{equation*}

We assume that there exists a dense subset $A\subseteq[0,T]$ containing $T$ and such that, for all $t\in A,g_N(t)$ converges to $g(t)$. Moreover, we assume that for all $i_1\neq i_2,$ for all $j_1\in\llbracket 1,n_{i_1-1}\rrbracket$ and $j_2\in\llbracket 1,n_{i_2-1}\rrbracket$, $t_{i_1,j_1}\neq t_{i_2,j_2}$. Then $g_N$ converges to $g$ in $D\left([0,T],\r^k\right)$.
\end{lem}

\begin{proof}
Since $g_N(T)=(n^N_i-1)_{1\leq i\leq k}$ converges to $g(T)=(n_i-1)_{1\leq i\leq k}$, we know that $n^N_i=n_i$ for all $N$ (large enough) and all $1\leq i\leq k$.

Now, we show that for each $1\leq i\leq k,1\leq j\leq n_i-1,$ $t^N_{i,j}$ converges to $t_{i,j}.$ As the sequence $(t^N_{i,j})_N$ is bounded, it is sufficient to show that $t_{i,j}$ is its only limit point. Let $s$ be a limit of a subsequence $(t^{\phi(N)}_{i,j})_N$.

We show that $s=t_{i,j}.$ If $s>t_{i,j}$ there would exist some $r\in A\cap]t_{i,j},s[$ satisfying that $g_{\phi(N)}(r)$ converges to $g(r).$ This is not possible because, as $r<s=\lim_N t_{i,j}^{\phi(N)},$ $g_{\phi(N)}(r)_i\leq j-1$ for $N$ large enough, and as $r>t_{i,j},$ $g(r)_i\geq j.$ For the same reason, it is not possible to have $s<t_{i,j}.$ As a consequence, $t_{i,j}$ is the only limit point of the bounded sequence $\ll(t_{i,j}^N\rr)_N$. This implies the convergence of $t^N_{i,j}$ to $t_{i,j}.$ In the rest of the proof, let us re-index the set $\{t_{i,j}:1\leq i\leq k,1\leq j\leq n_i-1\}$ as $\{s_i:1\leq i\leq n\}$ where $n=\sum_{i=1}^k(n_i-1),$ such that $s_1<s_2<...<s_n.$ And we consider the same indexes for the points $t^N_{i,j}$ ($1\leq i\leq k,1\leq j\leq n_i-1$).

To prove the convergence of $g_N$ to $g$ in $D([0,T],\r)$, we just have to define the sequence of functions $(\lambda_N)_N$ such that each $\lambda_N$ is the function that is linear on each interval $[s_i^N,s^N_{i+1}]$ and that satisfies $\lambda_N(s^N_i)=s_i.$ These functions verify $g_N=g\circ\lambda_N$ and
$$||\lambda_N-Id||_{\infty,[0,T]}=\underset{1\leq i\leq n}{\max}|s_i-s_i^N|=\underset{\substack{1\leq i\leq k\\1\leq j\leq n_i-1}}{\max}|t_{i,j}-t^N_{i,j}|\underset{n\rightarrow\infty}{\longrightarrow}0.$$
\end{proof}

Now, we can give the

\begin{proof}[Proof of Theorem~\ref{continuitypointprocess}]
Let $(x^k,\pi^k)=(x^{1,k},\hdots,x^{m,k},\pi^{1,k},\hdots,\pi^{m,k})_k$ converges in $D(\r_+,\r)^m\times\mathcal{N}^m$ to $(x,\pi)=(x^1,\hdots,x^m,\pi^1,\hdots,\pi^m)$. Let $Z:=\Phi(x,\pi)$ and $Z^k:=\Phi(x^k,\pi^k).$

Let us consider $T\geq 0$ such that for all $1\leq j\leq m,$ $\pi^j(\{T\}\times\r_+)=0$ and for all $k\in\n^*,\pi^{j,k}(\{T\}\times\r_+)=0$. In particular $T$ is a point of continuity of $Z$ and of each $Z^k$, and, as the set of all the atoms of the measures $\pi^j$ and $\pi^{j,k}$ ($1\leq j\leq m, k\geq 1$) is countable, the set of points~$T$ satsifying the previous conditions is dense.

According to the proof of Theorem~16.2 of \cite{billingsley_convergence_1999}, in order to prove the convergence of $Z^{k}$ to $Z$ in $D(\r_+,\r^m)$, it is sufficient to prove this convergence in $D([0,T],\r^m)$ for the points $T$ satisfying the conditions of the previous paragraph. Indeed, these points $T$ are continuity points of $Z$ and there exists an increasing sequence of such points $T$ going to infinity. Then, by Lemma~\ref{critereconvergence} (whose hypothesis is satisfied thanks to hypothesis~$(a)$ and~$(b)$), the convergence of $Z^k$ to $Z$ in $D([0,T],\r^m)$ will follow from the convergence of $Z^k_t$ to $Z_t$ for every point~$t$ satsifying the same conditions as~$T$. Let us show the convergence of $Z^{j,k}_t$ to $Z^j_t$ for every $1\leq j\leq m.$ In the rest of the proof, we work with fixed~$j,t,T.$

To show this, fix some $M>\max(\left|\left|x^j\right|\right|_{\infty,[0,T]},{\sup}_k~\left|\left|x^{j,k}\right|\right|_{\infty,[0,T]})$ (where we know that the supremum over $k$ of $\left|\left|x^{j,k}\right|\right|_{\infty,[0,T]}$ is finite since $(x^{j,k})_k$ converges in Skorohod topology) such that $\pi^j(\r_+\times[0,M])=0$, and write $\left\{\left(\tau_{i},\zeta_{i}\right)~:~1\leq i\leq N\right\}$ the set of the atoms of $\pi^j_{|[0,t]\times[0,M]}$ and $\left\{\left(\tau_{i}^k,\zeta_{i}^k\right)~:~1\leq i\leq N_k\right\}$ that of $\pi^{j,k}_{|[0,t]\times[0,M]}$.

Then, Proposition~\ref{cvvague} implies that $N_k=N$ for all $k$ (large enough), and that, for each $1\leq i\leq N,$ $\tau_i^k$ and $\zeta_i^k$ converge respectively to $\tau_i$ and $\zeta_i$ (we can assume that $\sigma^k=Id$ in the statement of Proposition~\ref{cvvague}, possibly reordering the indexes of the atoms of every $P^k$).

Notice that $$Z^{j,k}_{t}=\sum_{i=1}^{N}\un_{\left\{\zeta_{i}^k\leq x^{j,k}_{\tau_{i}^k-}\right\}}\un_{\left\{\tau_{i}^k\leq t\right\}}.$$

To end the proof, one has to note that $\un_{\{\zeta_{i}^k\leq x^{j,k}_{\tau_{i}^k-}\}}$ converges to $\un_{\{\zeta_{i}\leq x^j_{\tau_{i}-}\}},$ and that $\un_{\{\tau_{i}^{k}\leq t\}}$ converges to $\un_{\{\tau_{i}\leq t\}}=1.$ By hypothesis~$(d)$, $\zeta_i\neq x^j_{\tau_i-}$, whence there are two cases, either $\zeta_i<x^j_{\tau_i-}$ or $\zeta_i>x^j_{\tau_i-}.$ In the first case, we consider $\eps>0$ such that $\zeta_i+\eps<x^j_{\tau_i-}.$ Then, noticing that hypothesis~$(c)$ guarantees that $x^j$ is continuous at~$\tau_i$, Lemma~\ref{convergenceponctuel} and the convergence of $\tau_i^k$ and $\zeta_i^k$ respectively to $\tau_i$ and $\zeta_i$ imply that, for $k$ large enough, $\zeta^k_i<\zeta_i+\eps/3<x^j_{\tau_i}-\eps/3<x^{j,k}_{\tau^k_i-},$ what implies the convergence of $\un_{\{\zeta_{i}^k\leq x^{j,k}_{\tau_{i}^k-}\}}$ to $\un_{\{\zeta_{i}\leq x^j_{\tau_{i}-}\}}$. The second case is handled in the same way, as well as the convergence of $\un_{\{\tau_{i}^{k}\leq t\}}$ to $\un_{\{\tau_{i}\leq t\}},$ recalling that $\pi^j(\{t\}\times\r_+)=0,$ and so $\tau_i<t$.
\end{proof}

Let us finally prove our main result.

\begin{proof}[Proof of Theorem~\ref{convergencepointprocess}]
{\it Step 1:} Let us show that $(Z^{N,k})_{1\leq k\leq n}$ converges to $(\bar Z^k)_{1\leq k\leq n}$ as $N$ goes to infinity in $D(\r_+,\r^n).$

Since $(D(\r_+,\r)\times \mathcal{N})^{n}$ is a separable metric space (see Theorem~16.3 of \cite{billingsley_convergence_1999}, and Theorem~A2.6.III.(i) of \cite{daley_introduction_2003}), we can apply Skorohod representation theorem (see e.g. Theorem~6.7 of \cite{billingsley_convergence_1999}) to show the almost sure convergence of a sequence $((\tilde Y^{N,1},\tilde \pi^{N,1}),\hdots,(\tilde Y^{N,n},\tilde \pi^{N,n}))$ to $((\tilde Y^1,\tilde \pi^1),\hdots,(\tilde Y^n,\tilde \pi^n))$ in $(D(\r_+,\r)\times\mathcal{N})^{n}$ as $N$ goes to infinity, where these variables have respectively the same distribution as $((Y^{N,1},\pi^1),\hdots,(Y^{N,n},\pi^n))$ and $((\bar{Y}^1,\bar{\pi}^1),\hdots,(\bar{Y}^n,\bar{\pi}^n))$.

Then Theorem~\ref{continuitypointprocess} implies the almost sure convergence of the multivariate point processes $(\widetilde{Z}^{N,k})_{1\leq k\leq n}:=\Phi((\widetilde{Y}^{N,k},\widetilde{\pi}^{N,k})_{1\leq k\leq n})$ to $(\widetilde{Z}^{k})_{1\leq k\leq n}:=\Phi((\widetilde{Y}^k,\widetilde{\pi}^k)_{1\leq k\leq n})$ in $D(\r_+,\r^n)$. Let us show that the hypothesis of Theorem~\ref{continuitypointprocess} are satisfied almost surely. Hypothesis~$(a)$ is a classical property of Poisson measures (see Lemma~\ref{auplusun}). Hypothesis~$(b)$ is satisfied since the Poisson measures $\bar \pi^i$ ($i\geq 1$) are independent, whence, considering $i\neq j,$ denoting $A(\bar \pi^j)$ the set of points $t\geq 0$ such that $\bar\pi^j(\{t\}\times\r_+)=0,$ we have that $\bigcup_{t\in A(\bar\pi^j)}\{t\}\times\r_+$ is a null set (since $A(\bar\pi^j)$ is finite or countable) independent of $\bar \pi^i,$ and consequently
$$\pro{\bar \pi^i(\bigcup_{t\in A(\bar\pi^j)}\{t\}\times\r_+)\neq 0}=\esp{\pro{\bar\pi^i(\bigcup_{t\in A(\bar\pi^j)}\{t\}\times\r_+)\neq 0|\bar\pi^j}}=0,$$
by Lemma~\ref{mesurenulle}.

Hypothesis~$(c)$ and~$(d)$ are satisfied for a similar reason. For $(c)$, one has to observe that
$$\pro{\exists t>0,\pi^j(\{t\}\times\r_+)=1\textrm{ and }\bar Y^j\textrm{ is not continuous at }t}=\pro{\bar \pi^j(\bigcup_{t\in D(\bar Y^j)}\{t\}\times\r_+)\geq 1},$$
where $D(\bar Y^j)$ is the set of discontinuity points of $\bar Y^j$. As $D(\bar Y^j)$ is a.s. finite or countable (see e.g. the discussion after Lemma~1 of Section~12 of \cite{billingsley_convergence_1999}), whence $\bigcup_{t\in D(\bar Y^j)}\{t\}\times\r_+$ is a null set independent of $\bar\pi^j$, and Lemma~\ref{mesurenulle} gives the result:
$$\pro{\bar \pi^j(\bigcup_{t\in D(\bar Y^j)}\{t\}\times\r_+)\geq 1}=\esp{\pro{\bar \pi^j(\bigcup_{t\in D(\bar Y^j)}\{t\}\times\r_+)\geq 1|\bar Y^j}}=0.$$
And hypothesis~$(d)$ holds true, because the set $\{(t,\bar Y^j_{t-}):t\geq 0\}$ is also a null set independent of~$\bar\pi^j$.

Then, the almost sure convergence of $(\widetilde{Z}^{N,k})_{1\leq k\leq n}$ to $(\widetilde{Z}^{k})_{1\leq k\leq n}$ implies the convergence in distribution of $({Z}^{N,k})_{1\leq k\leq n}$ to $({Z}^{k})_{1\leq k\leq n}$ in $D(\r_+,\r^n).$

{\it Step 2:} We have shown that, for every $n\in\n^*,$ $(Z^{N,k})_{1\leq k\leq n}$ converges to $(\bar Z^k)_{1\leq k\leq n}$ in distribution in $D(\r_+,\r^n).$ This implies the convergence in the weaker topology $D(\r_+,\r)^n.$ Then, the convergence of $(Z^{N,k})_{k\geq 1}$ to $(\bar Z^k)_{k\geq 1}$ as $N$ goes to infinity in $D(\r_+,\r)^{\n^*}$ is classical (see e.g. Theorem~3.29 of \cite{kallenberg_foundations_1997}).
\end{proof}

%

\bibliography{biblio}

\providecommand{\bysame}{\leavevmode\hbox to3em{\hrulefill}\thinspace}
\providecommand{\MR}{\relax\ifhmode\unskip\space\fi MR }
\providecommand{\MRhref}[2]{%
  \href{http://www.ams.org/mathscinet-getitem?mr=#1}{#2}
}
\providecommand{\href}[2]{#2}
\begin{thebibliography}{10}

\bibitem{abi_jaber_weak_2019}
Eduardo Abi~Jaber, Christa Cuchiero, Martin Larsson, and Sergio Pulido, \emph{A
  weak solution theory for stochastic {Volterra} equations of convolution
  type}, arXiv:1909.01166 [math] (2019).

\bibitem{bauwens_modelling_2009}
Luc Bauwens and Nikolaus Hautsch, \emph{Modelling financial high frequency data
  using point processes}, Springer Berlin Heidelberg, 2009.

\bibitem{billingsley_convergence_1999}
Patrick Billingsley, \emph{Convergence of {Probability} {Measures}}, second
  ed., Wiley Series In Probability And Statistics, 1999.

\bibitem{bremaud_stability_1996}
Pierre Br{\'e}maud and Laurent Massouli{\'e}, \emph{Stability of {Nonlinear}
  {Hawkes} {Processes}}, The Annals of Probability \textbf{24} (1996), no.~3,
  1563--1588.

\bibitem{brown_martingale_1978}
Tim Brown, \emph{A {Martingale} {Approach} to the {Poisson} {Convergence} of
  {Simple} {Point} {Processes}}, The Annals of Probability \textbf{6} (1978),
  no.~4, 615--628.

\bibitem{daley_introduction_2003}
D.~J. Daley and D.~Vere-Jones, \emph{An {Introduction} to the {Theory} of
  {Point} {Processes}: {Volume} {I}: {Elementary} {Theory} and {Methods}},
  second ed., Springer, 2003.

\bibitem{delattre_hawkes_2016}
Sylvain Delattre, Nicolas Fournier, and Marc Hoffmann, \emph{Hawkes processes
  on large networks}, The Annals of Applied Probability \textbf{26} (2016),
  no.~1, 216--261.

\bibitem{erny_mean_2019}
Xavier Erny, Eva L{\"o}cherbach, and Dasha Loukianova, \emph{Mean field limits
  for interacting {Hawkes} processes in a diffusive regime}, arXiv:1904.06985
  [math] (2019).

\bibitem{ikeda_stochastic_1989}
Nobuyuki Ikeda and Shinzo Watanabe, \emph{Stochastic {Differential} {Equations}
  and {Diffusion} {Processes}}, second ed., North-Holland Publishing Company,
  1989.

\bibitem{jacod_limit_2003}
Jean Jacod and Albert~N Shiryaev, \emph{Limit {Theorems} for {Stochastic}
  {Processes}}, second ed., Springer-Verlag BerlinHeidelberg NewYork, 2003.

\bibitem{kallenberg_foundations_1997}
Olav Kallenberg, \emph{Foundations of {Modern} {Probability}}, Probability and
  {Its} {Applications}, Springer-Verlag, New York, 1997 (en).

\bibitem{lu_high_2018}
Xiaofei Lu and Fr{\'e}d{\'e}ric Abergel, \emph{High dimensional {Hawkes}
  processes for limit order books {Modelling}, empirical analysis and numerical
  calibration}, Quantitative Finance (2018), 1--16.

\bibitem{okatan_analyzing_2005}
Murat Okatan, Matthew~A. Wilson, and Emery~N. Brown, \emph{Analyzing functional
  connectivity using a network likelihood model of ensemble neural spiking
  activity}, Neural Computation \textbf{17} (2005), no.~9, 1927--1961 (eng).

\bibitem{reynaud-bouret_goodness--fit_2014}
Patricia Reynaud-Bouret, Vincent Rivoirard, Franck Grammont, and Christine
  Tuleau-Malot, \emph{Goodness-of-{Fit} {Tests} and {Nonparametric} {Adaptive}
  {Estimation} for {Spike} {Train} {Analysis}}, Journal of Mathematical
  Neuroscience \textbf{4} (2014), 3 -- 330325.

\end{thebibliography}

\end{document}